\documentclass[11pt, a4]{amsart}

\usepackage[all]{xy}

\usepackage{amsmath}
\usepackage{amsthm}
\usepackage{amssymb}
\usepackage{amscd}
\usepackage{subfigure}
\usepackage{epsfig}
\usepackage{color}
\usepackage{graphicx}

\def\N{\mathbb N}
\def\Z{\mathbb Z}

\def\R{\mathbb R}

\def\C{\mathbb C}

\def\A{\mathcal{A}}

\newcommand{\be}{\begin{eqnarray}}
\newcommand{\ee}{\end{eqnarray}}
\newcommand{\Tk}{{\mathcal T}}

\newcommand{\supp}{\mbox{\rm supp}}

\newcommand{\mmax}{{\rm max}}

\newcommand{\sgn}{{\rm sign}}

\newcommand{\Dk}{{\mathcal D}}
\newcommand{\Ak}{{\mathcal A}}

\newtheorem{theorem}{Theorem}[section]

\newtheorem{prop}[theorem]{Proposition}
\newtheorem{lem}[theorem]{Lemma}

\numberwithin{equation}{section}

\begin{document}
\title{Determining pure discrete spectrum for some self-affine tilings }


\date\today

\maketitle

\centerline{Shigeki Akiyama$^{\,\rm a}$, Franz G\"ahler$^{\,\rm b}$ and Jeong-Yup Lee$^{\,\rm c}$\footnote{Corresponding author. }}

\begin{abstract}
  By the algorithm implemented in Akiyama-Lee \cite{Akiyama-Lee:10}
  and some of its predecessors, we have examined the pure discreteness
  of the spectrum for all
  irreducible Pisot substitutions of trace less than
  or equal to $2$, and 
  some cases of planar tilings generated by boundary substitutions 
  due to Kenyon\cite{Kenyon:96}.
\end{abstract}

\section{Introduction}

Self-affine tilings are often studied as examples of tiling dynamics. 
Many equivalent conditions are known for the spectrum of the tiling dynamics 
to be pure discrete \cite{Pisot-Chapter}, but none of them is known to hold in 
general. For particular instances of tilings, there are algorithms by which 
one can check whether the spectrum is pure discrete. The overlap algorithm \cite{SS} 
and the balanced pair algorithm \cite{SS} are practically usable mostly in one 
dimension, but the potential overlap algorithm of Akiyama and Lee \cite{Akiyama-Lee:10} 
is of practical use in all dimensions, even if the tiles have complicated geometries.
Here we use these algorithms to check the pure discreteness of the spectrum of the 
tiling dynamics for special cases of self-affine tilings.
One of these cases is the 1-dimensional irreducible Pisot substitution tilings. 
There is a long-standing conjecture \cite{Pisot-Chapter} that these tilings have 
pure discrete spectrum. For the cases with two tiles, it is known that the 
conjecture is true \cite{HS03}, but not much is known for more than two tiles. 
Already the case with three tiles is computationally involved. We concentrate
here on the cases with substitution matrices of small trace. More precisely,
all inequivalent substitution matrices $M$ with $Tr(M)\le2$ have been generated,
and for all substitutions with these matrices (446683 substitutions in total), 
we have checked that the spectrum is pure discrete. 


The other type of tilings we consider are the self-affine tilings 
constructed from the endomorphisms of free groups by Kenyon
\cite{Kenyon:96}. We have looked at the cubic 
polynomials whose
coefficients are all less than or equal to $3$, 
except for a single case, whose
computation is beyond our computer capability. All other examples turn
out to be pure discrete.

We also give some non pure discrete
examples of self-affine tilings in \S \ref{Non}. 
The construction of the first one is due to Bandt \cite{Bandt:97}.
The second arises from $4$ interval
exchanges studied by Arnoux-Ito-Furukado.
Both satisfy the 
Pisot family condition, so that their translation
actions are not weakly mixing.
We provide the programs in \cite{AL-web,FGweb}.

\section{Pisot substitutions with small trace}

In this section, we wish to computationally confirm the Pisot substitution conjecture to be true for a class of simple, irreducible Pisot substitutions with three tiles. 
Consider a monoid $\A^*$ over finite alphabets $\A$ equipped with concatenation
and write the identity as $\epsilon$, the empty word. 
A {\it symbolic substitution} $\sigma$
is a non-erasing homomorphism of $\A^*$, defined by  
$\sigma(a)\in \A^+=\A^* \setminus$ \{$\epsilon$\} for $a\in \A$. 
The set $\A^{\Z}$ of two sided sequences is compact by the product topology of the discrete topology on $\A$. 
The substitution $\sigma$ acts naturally on $\A^{\Z}$ by $\sigma(\dots a_{-1}a_0a_1a_2\dots)=
\dots \sigma(a_{-1})\sigma(a_0)\sigma(a_1)\sigma(a_2)\dots$.
Let $M_{\sigma}$ be the {\it incidence matrix} $(|\sigma(j)|_i)_{ij}$ where $i,j\in \A$. Here
$|w|_j$ is the cardinality of $j$ appearing in a word $w\in \A^*$.
Denote by $\chi_{\sigma}$ the characteristic polynomial of $M_{\sigma}$.
The substitution $\sigma$ is {\it primitive} if $M_{\sigma}$ is primitive and
it is {\it irreducible} if $\chi_{\sigma}$ is 
irreducible\footnote{We always assume the irreducibility of $M_{\sigma}$ in the sense of Perron-Frobenius theory. 
So the irreducibility in this article is for $\chi_{\sigma}$.
}. 
{\it A Pisot number} is an algebraic integer $\lambda > 1$ whose all the other algebraic conjugates of $\lambda$ lie strictly inside the unit circle.
If the Perron-Frobenius root of $M_{\sigma}$
is a Pisot number then we say that $\sigma$ is a Pisot substitution. 
A word $w\in \A^*$ is {\it admissible} if there exist $k\in \N$ and $a\in \A$
such that $w$ is a subword of $\sigma^k(a)$.
Let 
$$
X_{\sigma}=
\{ (a_n)_{n\in \Z}\in \A^{\Z}\ |\ a_{k}a_{k+1}\dots a_{\ell} 
\text{ is admissible for all } k, \ell \text{ with } k<\ell \}.
$$
Then $(X_{\sigma}, s)$ forms a topological dynamical system
where $s$ is the shift map 
defined by $s( (a_n)_{n\in \Z})=(a_{n+1})_{n\in \Z}$. 
By the primitivity, the system is minimal and uniquely ergodic with the
unique invariant measure $\mu$. Therefore we can discuss the spectrum 
of the unitary operator $U_{\sigma}$ acting on $L^2(X_{\sigma}, \mu)$ for which
$(U_{\sigma}(f))(x)=f(s(x))$.
The substitution $\sigma$ has {\em pure discrete dynamical spectrum} if 
the spectral measure associated to $U_{\sigma}$ consists only of point spectra,
 or equivalently, the linear span of eigenfunctions is dense in $ L^2(X_{\sigma},\mu)$.
It is conjectured \cite{Pisot-Chapter} that this $\Z$-action by $U_{\sigma}$ is pure discrete if 
$\sigma$ is an irreducible Pisot substitution -- so called {\it Pisot substitution conjecture}. 
For the primitive substitution $\sigma$, 
we can also discuss a natural suspension of $(X_{\sigma},s)$ by
associating to each letter the length determined by the associated 
entries of the left eigenvector of $M_{\sigma}$. The sequence
then defines a 
tiling of $\R$ with an inflation matrix $Q=(\beta)$, 
where $\beta$ is the Perron-Frobenius root of $M_{\sigma}$.
This gives a tiling dynamical system $(X_{\mathcal{T}},\R)$ which is
also minimal and uniquely ergodic. It is known \cite{Clark-Sadun:03} that 
if $\sigma$ is an irreducible Pisot substitution, 
this $\R$-action on $(X_{\mathcal{T}},\R)$ is
pure discrete if and only if the $\Z$-action on $(X_{\sigma},s)$ is
pure discrete. 

\medskip

The following assertion may be known but we did not find it in the
literature. It gives a bound for the number of irreducible primitive
substitutions over $m$ letters.

\begin{lem}
\label{Bound}
Let $B > 0$. The cardinality of the set of
primitive substitutions over $m$ letters, 
whose Perron Frobenius root is less than or equal to $B$, is less than 
$m^{m^4 B^{2(m-1)^2+2}}$.
\end{lem}

\begin{proof}
Let $\sigma$ be a substitution over $m$ letters whose incidence 
matrix is $M_{\sigma}$. Then there is a positive integer $k$ that
$M_{\sigma}^k=M_{\sigma^k}$ is a positive matrix. In fact, one can take 
$k \le (m-1)^2+1$ for all primitive matrix $M_{\sigma}$ 
(see \cite{HolladayVarga:58}, \cite{Ptak:58}).
Denote the characteristic polynomial by $\Phi_{\sigma^k}(x)=x^m-c_{m-1}x^{m-1}-c_{m-2}x^{m-2}
-\dots -c_0$. Then we have a bound
$|c_{m-i}|< {m \choose i} \beta^{ki}$, because other roots of $\Phi_{\sigma^k}$ are
less than $\beta^k$ in modulus. Our aim is to show that 
there are only finitely many matrices $M_{\sigma}^k=(a_{ij})$. Indeed this implies
that all entries of $M_{\sigma}$ are bounded by $\max \{ a_{ij} \ | \ i, j \le m \}$, since
otherwise there is an entry of $M_{\sigma^k}$ larger than this bound.
From 
$0\le \sum_i a_{ii}= c_{m-1}\le m \beta^k$, we have $a_{ii}< m \beta^k$
and it suffices to show that $a_{ij}$ is bounded for $i\neq j$.
Using 
$$
{m \choose 2}\beta^{2k}> c_{m-2}= \sum_{i<j} a_{ij}a_{ji} - \sum_{i<j} a_{ii}a_{jj}
$$
and
$$
0<\sum_{i<j} a_{ii}a_{jj} = \frac 12 \left(
\left(\sum_{i} a_{ii} \right)^2
- \sum_{i} a_{ii}^2 \right) \le \frac {m^2}2 \beta^{2k} 
$$
we see that
$$
0\le a_{ij}a_{ji}< \sum_{i<j} a_{ij}a_{ji} \le {m \choose 2} \beta^{2k} + \frac{m^2}2 \beta^{2k}
\le m^2 \beta^{2k}
$$
Since $a_{ji} \in \N$, we have the bound $a_{ij}\le  m^2 B^{2k}$. 
Thus we have $b_{ij} \le m^2 B^{2k}$ where $M_{\sigma}=(b_{ij})$ and the 
number of possible $\sigma(i)$'s for each $i$ is bounded by the multinomial coefficient:
\begin{eqnarray}
{m^3 B^{2k} \choose m^2 B^{2k}, m^2 B^{2k}, \dots, m^2 B^{2k}} 
\le m^{m^3 B^{2k}}
\le m^{m^3 B^{2(m-1)^2+2}}
\end{eqnarray}
which gives the required bound.
\end{proof}

Although Lemma \ref{Bound} gives a bound for the number of substitutions whose Perron-Frobenius root is less than $B$, it is too large to 
be useful in practice.
In the sequel, we deduce a practical estimate by 
the property of the Pisot number to narrow the range of
computation.
If $\sigma$ is a Pisot substitution of degree $d$ whose incidence matrix has a Pisot number 
$\beta$ as the Perron-Frobenius root,
we have $\beta-(d-1)< {\rm Tr} (M_{\sigma}) < \beta +(d-1)$.
Thus it is meaningful to check the Pisot conjecture for irreducible substitutions
whose incidence matrix $M_{\sigma}$ has small trace with a fixed degree. 
Our first result is the following proposition.

\begin{prop}
Let $\sigma$ be a primitive, irreducible, cubic Pisot substitution with 
${\rm Tr}(M_{\sigma}) \le 2$. Then the spectrum of $\sigma$ is pure discrete.
\end{prop}

\begin{proof}
Note that ${\rm Tr}(M_{\sigma})\ge 0$ since $M_{\sigma}$ is a non-negative matrix.
Let 
\begin{eqnarray} \label{Char-Polynomial}
x^3-px^2-q x-r, \ \ \ \mbox{where} \ p, q, r \in \Z
\end{eqnarray}
be the characteristic polynomial of $\sigma$. It 
is the minimal polynomial of a cubic Pisot number if and only if
\begin{equation}
\label{Pisot}
\mmax \{2-p-r,\ r^2-\sgn(r)(1+pr)+1\} \le q \le p+r \text{ and } r\neq 0
\end{equation}
holds (see \cite{Akiyama-Gjini:04}). 
We see $p = {\rm Tr}(M_{\sigma}) \ge 0$.
Let
$$
M_{\sigma}:=
\begin{pmatrix}
k_1& a & b\\
c& k_2 & d\\
e& f & k_3\\
\end{pmatrix} 
$$
be the incidence matrix of the primitive substitution $\sigma$.
Then the coefficients $p, q, r$ of the 
characteristic polynomial (\ref{Char-Polynomial}) can be written as
\begin{eqnarray*}
p&=&k_1+k_2+k_3 \ \ \mbox{and} \ \ p\ge 0, \\ 
q&=&ac+be+df-k_1k_2-k_2k_3-k_1k_3, \\ 
r&=&k_1k_2k_3+dcf+ade-dfk_1-bek_2-ack_3.
\end{eqnarray*}
We claim that $a,b,c,d,e,f$ are not greater than $L$, where
\begin{equation}
\label{Estimate}
L := \max\{q+k_1k_2+k_2k_3+k_3k_1,\ 
r-k_1k_2k_3 + \max\{k_1,k_2,k_3\} (q+k_1k_2+k_2k_3+k_1k_3) \}
\end{equation}
using the idea of Lemma \ref{Bound}.
Note that by the primitivity, none of six vectors 
$(a,b)$, $(c,d)$, $(e,f)$, $(c,e)$, $(a,f)$, $(b,d)$ is $(0,0)$.
By symmetry, we only prove this bound (\ref{Estimate}) for $a$. 
Consider $q$ and $r$ as a linear polynomial on $a$. Then the leading 
coefficients are $c$ and $de-ck_3$ and we see 
that either $c\neq 0$ or $de-ck_3\neq 0$ holds. 
If $c\neq 0$, then the formula
for $q$ gives
$a\le ac \le q+k_1k_2+k_2k_3+k_1k_3$. If $c=0$ then, the formula for $q$
implies $be+df\le q+k_1k_2+k_2k_3+k_1k_3$ and the formula for $r$ gives
$$a\le ade \le r-k_1k_2k_3 + df k_1+ be k_2 \le 
r-k_1k_2k_3 + \max\{k_1,k_2\} (q+k_1k_2+k_2k_3+k_1k_3).$$
So the claim follows.
Since $k_1,k_2,k_3$ are non-negative integers, from (\ref{Estimate}) one
can also deduce $a,b,c,d,e,f \le r+p(q+p^2)$.

In the case $p=0$ we easily have $q=r=1$ from (\ref{Pisot}). By using 
the bound of $a,b,c,d,e,f$, there are
$6$ substitutions matrices with the characteristic polynomial $x^3-x-1$.  
However, these matrices form a single orbit under the group $S_3$ which 
permutes the symbols, so that only one matrix with 2 substitutions
needs to be checked. Moreover, there is a further symmetry which
can be taken into account in order to avoid duplicate work.
Let $\xi:\A^* \rightarrow \A^*$ be the `mirror' map, i.e.,
$\xi(a_1a_2\dots a_{n-1}a_n)=
a_na_{n-1}\dots a_2a_1$. For a given substitution $\sigma$, we define
$\overline{\sigma}$ by $\overline{\sigma}(a)=\xi(\sigma(a))$ for $a\in \A$.
Then the dynamical systems $X_{\sigma}$ and $X_{\overline{\sigma}}$ are
clearly isomorphic and we only have to compute one of them. 
By using this symmetry, the two substitutions of the $p=0$ case reduce to one,
which turns out to be pure discrete. 
Next we study the case $p>0$.

Changing the order of letters, for $p=1$ we may assume
that the incidence matrix $M_{\sigma}$ is of the form:
$$
\begin{pmatrix}
1& a & b\\
c& 0 & d\\
e& f & 0\\
\end{pmatrix} \text{ with } a\ge b
$$
and for $p=2$ we have
$$
\begin{pmatrix}
2& a & b\\
c& 0 & d\\
e& f & 0\\
\end{pmatrix} \text{ with } a\ge b, \text{ or }
\begin{pmatrix}
1& a & b\\
c& 1 & d\\
e& f & 0\\
\end{pmatrix} \text{ with } a\ge c.
$$
For each case we can deduce bounds for $a,b,c,d,e,f$ 
from the inequality (\ref{Estimate}). The list of matrices
satisfying these bounds is finite, but still contains some 
pairs which are equivalent under a permutation of the symbols. 
Only one matrix of each such pair is retained. For any substitution 
matrix $M$ in the list, there are only finitely many substitutions
$\sigma$ with $M_{\sigma}=M$ (of which we keep only one per mirror 
pair), so that we obtain a finite list of irreducible Pisot 
substitutions to be checked. 
Then we use the `potential overlap algorithm' of \cite{Akiyama-Lee:10} 
and a new implementation \cite{FGweb} of the classical overlap algorithm 
\cite{SS, soltil}
to check all the substitutions in the list. Moreover, since the 
substitutions in the list are irreducible (Prop. 1.2.8 of \cite{Fogg}), 
we may check by the balanced pair algorithm \cite{SS} as well.
Up to symbol relabelling and mirror symmetry as above, there are 
7377 irreducible cubic Pisot unit substitutions, as given in Table 1, 
and all of them are pure discrete. The number of non-unit substitutions,
especially with trace $p=2$, is much larger, and on average each 
computation is much harder. These non-unit cases could mostly be 
checked only with a new implementation \cite{FGweb}
of the classical overlap algorithm in the GAP language 
\cite{GAP}, in which a special effort 
has been made to keep the memory requirements small. Also the non-unit 
substitutions turned out to be pure discrete. The results are summarized 
in Table~1.
\begin{table}[h]
\label{Parameter}
\begin{tabular}{|c|c|c||c|c|c|}
\hline
$p$ & $r$ & $q$ & $M_{\sigma}$ & $\sigma$ & Spectrum \\
\hline
0 & 1    & $1$          & 1     & 1 & pure discrete\\
1 & 1    & $0,1,2$      & 12    & 61 & pure discrete\\
1 & 2    & $2,3$        & 30    & 457 & pure discrete \\
2 & $-$1 & $1$          & 1+4   & 3+36 & pure discrete\\
2 & 1    & $-1,0,1,2,3$ & 22+46 & 1077+6199 & pure discrete\\
2 & 2    & $0,1,2,3,4$  & 72+97 & 10586+46348 & pure discrete\\
2 & 3    & $3,4,5$      & 59+93 & 65383+316532 & pure discrete\\
\hline
\end{tabular}
\vspace{5mm}
\caption{
The columns labelled $M_{\sigma}$ and $\sigma$ contain the number 
of inequivalent substitution matrices and substitutions, respectively.
In the trace~2 case, these numbers are split into the contributions
from matrices with diagonal $(2,0,0)$ and $(1,1,1)$, respectively.}
\end{table}

\end{proof}



In principle, the above method of listing all cubic Pisot substitutions
with a fixed trace can be used also for matrices $M_\sigma$ with ${\rm Tr}(M_{\sigma})= 3$. 
However, for this case there would be far too many substitutions to be checked, 
which currently seems beyond reach.

Experiments with an implementation of the balanced pair algorithm in GAP \cite{FGweb}, 
similar to that of the classical overlap algorithm \cite{SS, soltil}, suggests that the 
latter has an advantage for more complex substitutions. This is likely due
to the balanced pairs not only growing in number, but also in length,
which requires more memory to store them, but especially also more work
to compare them. The overlap types needed in the overlap algorithm,
on the other hand, grow only in number, and each of them requires
only a fixed (small) amount of memory.


\section{Substitution tilings in $\R^2$}



In this section, we check the pure discreteness of
self-affine tilings constructed by Kenyon \cite{Kenyon:96}.
Before explaining his construction, we start with general notations.

A {\em tile} in $\R^2$ is defined as a pair $T=(A,i)$ where $A=\supp(T)$
(the support of $T$) is a compact
set in $\R^2$ which is the closure of its interior, and
$i=l(T)\in \{1,\ldots,m\}$
is the color of $T$. 
We say that
a set $P$ of tiles is a {\em patch} if the number of tiles in $P$ is
finite and the tiles of $P$ have mutually disjoint
interiors. A {\em tiling} of $\R^2$ is a set $\Tk$ of tiles such that $\R^2 = \bigcup \{\supp(T) : T \in \Tk \}$ and distinct tiles have disjoint interiors.

Let $\Ak = \{T_1,\ldots,T_m\}$ be a finite set of tiles in
$\R^2$ such that $T_i=(A_i,i)$; we will call them {\em
prototiles}. Denote by ${\mathcal{P}}_{\Ak}$ the set of non-empty
patches. 
A substitution is a map $\Omega: \Ak \to {\mathcal{P}}_{\Ak}$ with a
$2 \times 2$ expansive matrix $Q$ if there exist finite sets $\Dk_{ij}\subset
\R^2$ for $i,j \le m$ such that
\begin{equation}
\Omega(T_j)=
\{u+T_i:\ u\in \Dk_{ij},\ i=1,\ldots,m\}
\label{subdiv}
\end{equation}
with
\begin{eqnarray} \label{tile-subdiv}
Q A_j = \bigcup_{i=1}^m (\Dk_{ij}+A_i) \ \ \ \mbox{for} \ j\le m.
\end{eqnarray}
Here all sets in the right-hand side must have disjoint interiors;
it is possible for some of the $\Dk_{ij}$ to be empty.

We say that $\Tk$ is a {\em substitution tiling} if $\Tk$ is a
tiling and $\Omega(\Tk) = \Tk$ with some substitution $\Omega$. 
We say that $\Tk$ has {\em finite local complexity} (FLC) 
if $\forall \ R > 0$, $\exists$ finitely many translational classes of patches whose support lies in some ball of radius $R$.
A tiling $\Tk$ is {\em repetitive} if for any compact set
$K \subset \R^2$, $\{t \in \R^2 : \Tk \cap K = (t + \Tk) \cap K\}$ is relatively dense.
A repetitive fixed point of a primitive substitution with FLC is called a {\em self-affine tiling}.
Let $\lambda>1$ be the Perron-Frobenius eigenvalue  of the substitution matrix $S$.
Let
$ D =\{\lambda_1,\ldots,\lambda_{d}\}$ be the set of (real and complex)
eigenvalues of $Q$. 
We say that $Q$ (or the substitution $\Omega$) fulfills the {\em Pisot family} property if, for every $\lambda \in
D$ and every Galois conjugate $\lambda'$ of $\lambda$, 
$\lambda' \not \in D$, then $|\lambda'| < 1$.

\subsection{Endomorphisms of free group}

Generalizing the idea of Dekking \cite{Dekking:82, Dekking:82_2}, 
Kenyon \cite{Kenyon:96} 
introduced a class of self-similar tilings generated by the endomorphisms of
a free group over three letters $a,b,c$:
\begin{eqnarray*}
\theta(a)&=&b\\
\theta(b)&=&c\\
\theta(c)&=&c^{p}a^{-r}b^{-q},
\end{eqnarray*}
where $p,q\ge 0, r\ge 1$ are integers for which
$x^3-p x^2+qx+r$ has exactly two roots $\lambda_1, \lambda_2$ with modulus
greater than one.\footnote{Kenyon \cite{Kenyon:96} studied the cases when
$\lambda_i$ are complex numbers.} 
The letters $a,b,c$ are identified with vectors 
$(1,1), (\lambda_1,\lambda_2)$, $(\lambda_1^2,\lambda_2^2)\in \R^2$ respectively
if $\lambda_1$, $\lambda_2$ are real numbers, and with
$1, \lambda_1, \lambda_1^2 \in \C$ if they are complex conjugates.
The endomorphism
$\theta$ acts naturally
on the boundary word 
\[ aba^{-1}b^{-1},aca^{-1}c^{-1}, bcb^{-1}c^{-1},\] 
which represent three fundamental parallelograms, and gives a substitution rule
on the parallelograms. 
The associated tile equations are
\begin{eqnarray*}
Q T_1&=&T_2\\
Q T_2&=&\left(\bigcup_{i=1}^q (T_2+ p c -i b -r a)\right) 
\cup 
\left(\bigcup_{i=1}^r (T_3+p c - i a) \right) \\
Q T_3&=&\left(\bigcup_{i=1}^r (T_1 +p c - i a) \right) \cup
\left(\bigcup_{i=1}^{p-1} (T_2 + i c) \right)
\end{eqnarray*}
where $Q$ is either $\begin{pmatrix} \lambda_1 & 0 \cr 0 & \lambda_2\end{pmatrix}$
or $\lambda_1$ depending on whether $\lambda_1$ is real or complex, respectively.

\begin{figure}
\begin{center}
\subfigure[Tiling]{
\includegraphics[clip, scale=0.7]{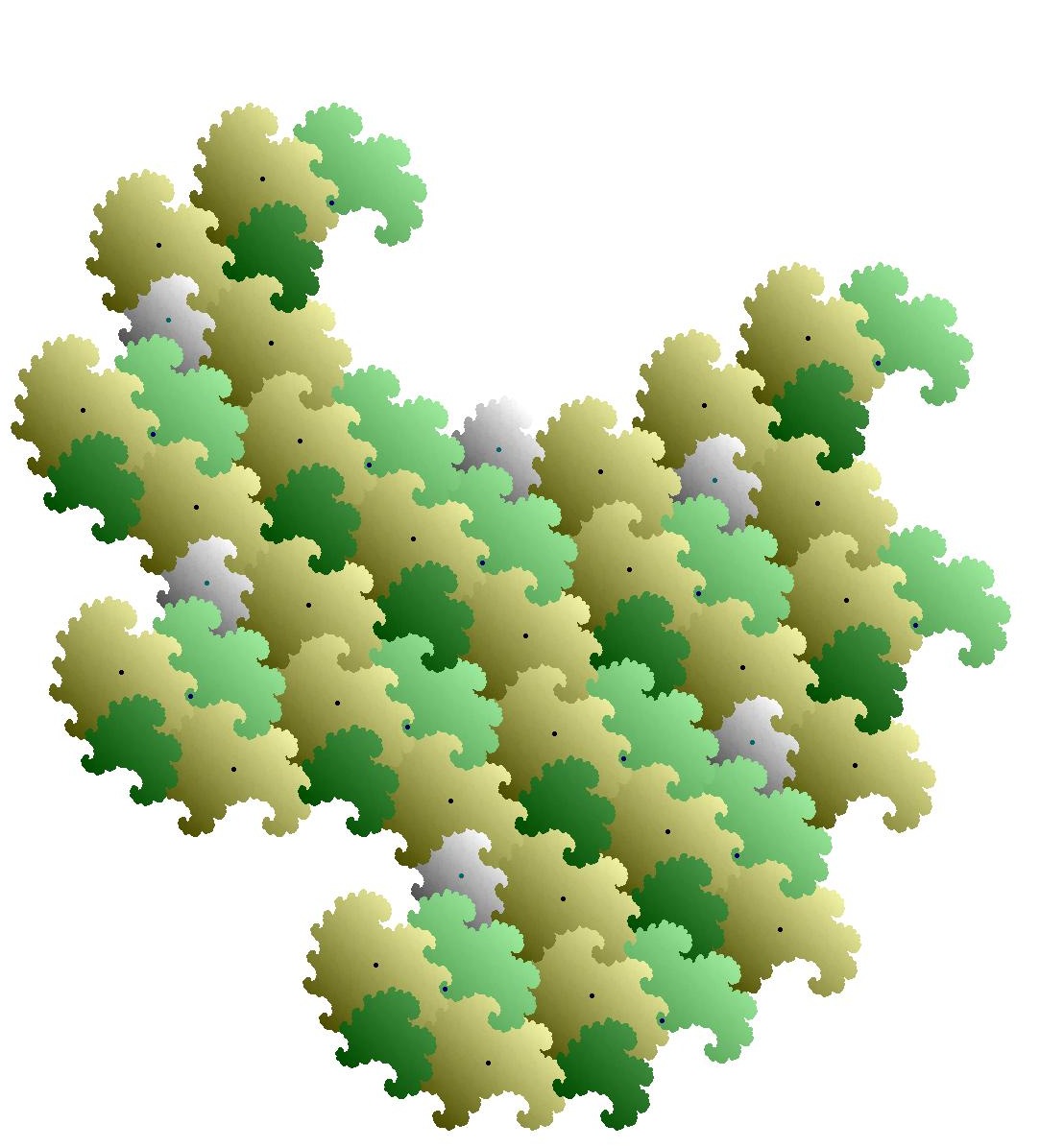}
}
\hspace{0.5cm}
\subfigure[Tiles]{
\includegraphics[clip, scale=0.7]{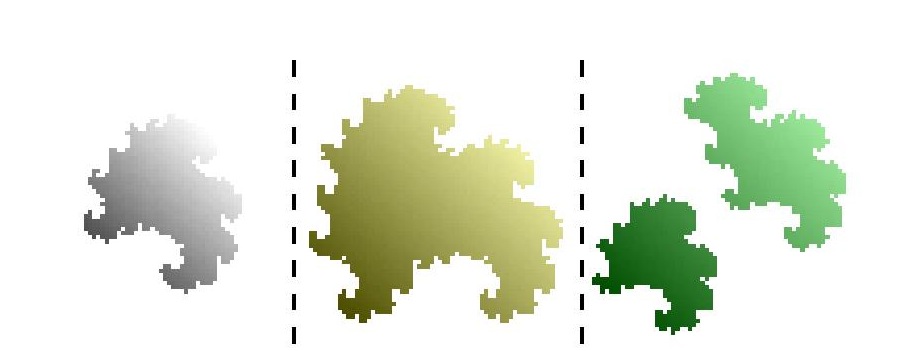}
}
\caption{Kenyon's tiling for $(p,q,r)=(2,1,1)$. 
The last tile is disconnected.} \label{Ex7_5}
\end{center}
\end{figure}

It is known \cite{LS, LS2} that if the expansion map $Q$ of a self-affine tiling in $\R^2$ is diagonalizable and the tiling has pure discrete dynamical spectrum, then $Q$ should fulfill the Pisot family property.
So we are interested in considering self-affine tilings with the Pisot family property on the expansion map $Q$. 
For this construction, 
we require that two roots of
the polynomial $x^3-p x^2+qx+r$ are greater than one, 
and one root is smaller than one in modulus.  In this case, we can note that there are no roots on the unit circle.
We adapt the Schur-Cohn criterion (see \cite[Theorem 2.1]{Akiyama-Gjini:04}
or \cite[Chap.10, Th.43, 1]{Marden:66}), which says that
the number of roots within the unit circle coincides with the
number of sign changes of the following sequence:
\[ 1, \Delta_1, \Delta_2, \Delta_3, \]
where 
\begin{eqnarray*}
\Delta_1 &=& - \left| \begin{array}{rr}
                      1 & r \\
                      r & 1 
                    \end{array}\right|= (r-1) (r+1), \\ 
\Delta_2 &=& \left| \begin{array}{rrrr}
                      1 & -p & r & 0  \\
                        0 &  1  &  q & r \\
                       r & q & 1 & 0 \\
                        0 & r & -p & 1 
                    \end{array}\right|
= -\left(p r+q-r^2+1\right) \left(p r+q+r^2-1\right), \\
\Delta_3 &=& -\left| \begin{array}{rrrrrr}
                      1 & -p & q & r & 0 & 0 \\
                       0  &  1  &  -p & q & r & 0 \\
                       0  &  0  & 1    & -p & q & r \\
                        r & q & -p & 1 & 0 & 0 \\
                       0  & r  & q & -p & 1 & 0 \\
                        0  & 0   &  r & q & -p & 1
                    \end{array}\right| \\
&=& (p-q-r-1)(p+q-r+1) \left(p r+q+r^2-1\right)^2,
\end{eqnarray*}
if all entries are non-zero. 
Notice that it cannot be that $p = q = 0$, otherwise 
all the roots of the polynomial have the same modulus.
Thus from $p, q \ge 0$ and $r \ge 1$,  $p r+q+r^2-1>0$.
So the signs come from 
$$
1, (r-1) (r+1), -(p r+q-r^2+1), (p-q-r-1)(p+q-r+1)
$$
The last term is not zero, since $\pm 1$ cannot be a root of $x^3-px^2+qx+r$.
When $r=1$, the second term vanishes and 
the third term is negative (because $p$ or $q$ is positive). 
Since the roots of the polynomial
are continuous with respect to the coefficients, 
the small perturbation of $r$ does not change the number of roots inside/outside
of the unit circle. 
Therefore we may assume that the second coefficients are non-zero
and use the Schur-Cohn criterion (c.f. \cite{Akiyama-Gjini:04}). 
As a result, the number of 
zeroes within the unit circle is
$1$ when $(p-q-r-1)(p+q-r+1)<0$ and it is $2$ when $(p-q-r-1)(p+q-r+1)>0$.
If $r>1$, then the second term is positive.  
In this case, 
applying the small perturbation argument when the third term vanishes,
the number of zeroes within the unit circle is
$1$ when $(p-q-r-1)(p+q-r+1)<0$ and $0$ or $2$ when $(p-q-r-1)(p+q-r+1)>0$.
Overall we obtain a unified 
conclusion that
the number of zeroes within the unit circle is
$1$ if and only if $(p-q-r-1)(p+q-r+1)<0$.
Our tiling exists when $|p-r|<q+1$.

For $\max \{p,q,r\}\le 3$, there are 34 cases:
$$
\{0, 1, 1\}, \{0, 2, 1\}, \{0, 2, 2\}, \{0, 3, 1\}, \{0, 3, 2\}, \{0, 3, 3\}, 
\{1, 0, 1\}, \{1, 1, 1\}, \{1, 1, 2\}, 
$$
$$
\{1, 2, 1\}, \{1, 2, 2\}, \{1, 2, 3\}, \{1, 3, 1\}, \{1, 3, 2\}, \{1, 3, 3\}, 
\{2, 0, 2\}, \{2, 1, 1\}, \{2, 1, 2\}, 
$$
$$
\{2, 1, 3\}, \{2, 2, 1\}, \{2, 2, 2\}, \{2, 2, 3\}, \{2, 3, 1\}, \{2, 3, 2\}, 
\{2, 3, 3\}, \{3, 0, 3\}, \{3, 1, 2\}, 
$$
$$
\{3, 1, 3\}, \{3, 2, 1\}, \{3, 2, 2\}, \{3, 2, 3\}, \{3, 3, 1\}, \{3, 3, 2\}, 
\{3, 3, 3\}
$$
We determined the spectral type of
the tiling dynamical systems except for the 
case $\{3,0,3\}$. 
All computed systems admit an overlap coincidence, and thus
have a pure discrete spectrum.
In the remaning case $\{3,0,3\}$, the tiles are very thin,
so that the number of initial overlaps seems to be
beyond the capability of our program.

\section{Examples with non pure discrete spectrum}
\label{Non}

In this section, we wish to give two intriguing examples 
of self-affine tilings 
whose dynamical spectrum is not pure discrete. Both tiling dynamics
are not weakly mixing and therefore there exist non-trivial eigenvalues 
of their translation actions. 

Bandt discovered a non-periodic tiling in \cite{Bandt:97}, whose setting
comes from crystallographic tiles, and called it the {\it fractal chair tiling}.
The tile satisfies the following set equation
$$
- I \omega \sqrt{3} A = A \cup (A+1) \cup (\omega A +\omega) 
$$
\noindent
where $\omega=(1+\sqrt{-3})/2$ is the 6-th root of unity. It is called 
$3$-rep-tile, because it is a non-overlapping
union of three similar
contracted copies, i.e., the associated iterated function system satisfies the open set
condition \cite{Bandt:97}.
\begin{figure}
\begin{center}
\includegraphics[clip, scale=1.0]{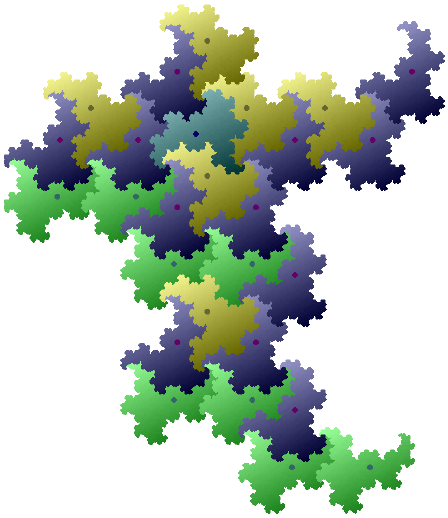}
\caption{Fractal chair tiling \label{Bandt}} 
\end{center}
\end{figure}
Applying the substitution rule, we obtain a tiling of the plane
by six tiles $T_i= \omega^i A\ (i=0,1,2,3,4,5)$ and their translates 
as given
in Figure \ref{Bandt}.
This tiling is {\it non-periodic}, i.e., the only translation
which sends the tiling exactly to itself is the zero vector.
We can confirm 
that this tiling is repetitive and the corresponding expansion map
satisfies the Pisot family property. 
So we can apply the potential overlap
algorithm from \cite{Akiyama-Lee:10}.
Our program shows that the fractal chair tiling is not purely discrete. 
The overlap graph with multiplicity contains 
a strongly connected component of spectral radius $3$, being equal to 
the spectral radius of the substitution matrix, 
which does not lead to 
a coincidence.  
In fact, the existence of such a component 
is shown in \cite{Akiyama-Lee:10} to be equivalent to having non-pure discrete spectrum.
Geometrically this means that there is 
a finite set $C$ of overlaps not containing any coincidence, 
which is mapped to itself under the substitution. 
We call such a set $C$ a {\it non-coincident component}.
We visualize in Figure \ref{Overlap} how each overlap in the
non-coincident component does not lead to a coincidence. Each figure represents an overlap
of two fractal chair tiles, and the support of the overlap is depicted in
thick color. The destinations of outgoing arrows show 
that we obtain three overlaps from each overlap by substitution.
\begin{figure}
\includegraphics[clip, scale=0.4]{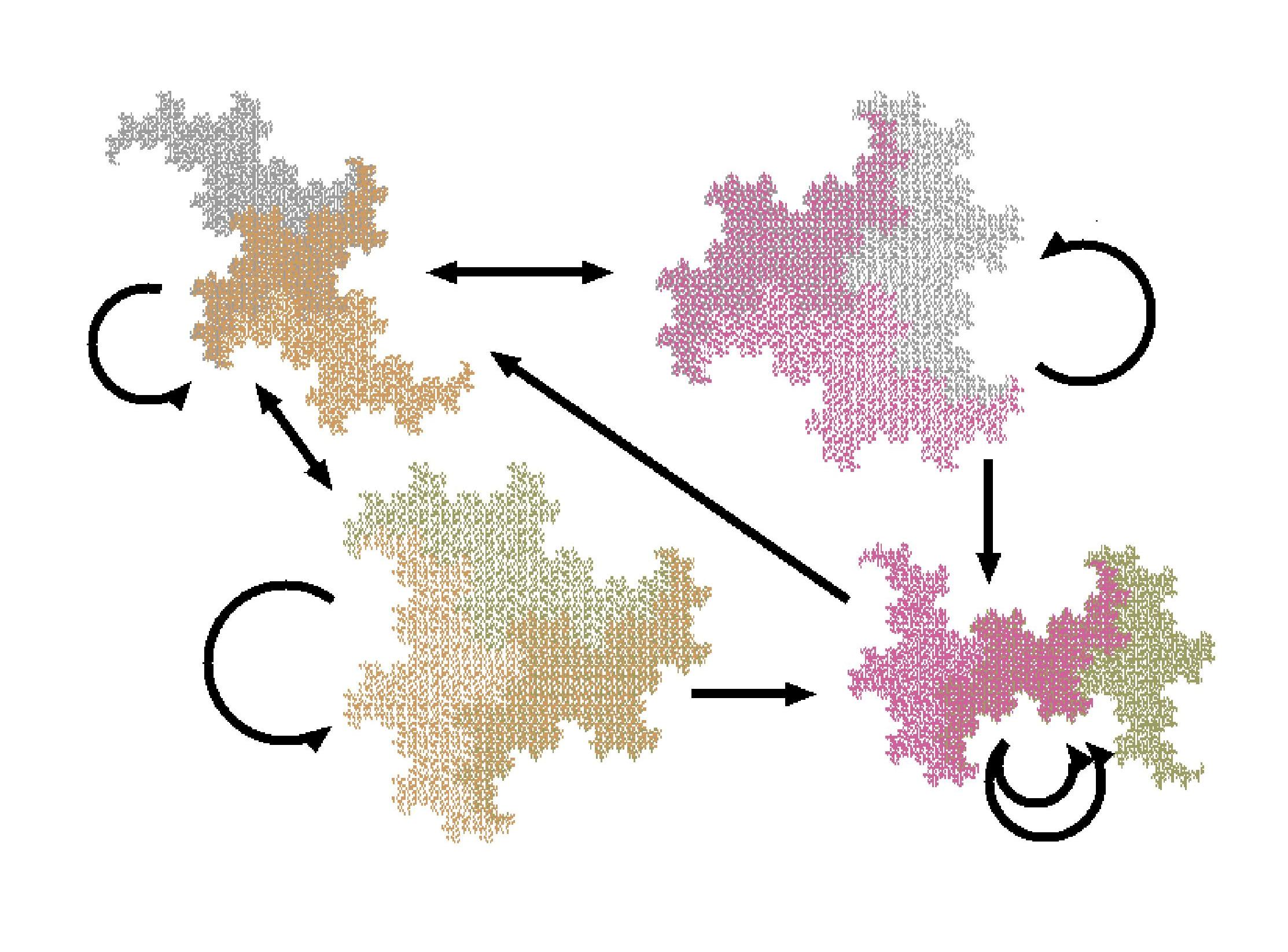}
\caption{Non-coincident component of fractal chair tiling \label{Overlap}} 
\end{figure}




As another example,
the following substitution 
\begin{eqnarray}
\label{4IET}
\nonumber
1&\rightarrow& 1241224,\\
2&\rightarrow& 1224,\\
\nonumber
3&\rightarrow& 1243334,\\
\nonumber
4&\rightarrow& 124334
\end{eqnarray}
with the characteristic polynomial $(x^2-3x+1)(x^2-6x+1)$ arose in the study of 
self-inducing 4 interval exchanges by P.~Arnoux, S.~Ito and M.~Furukado 
\cite{Ito_Kyoto:10}.  
As it is a Pisot substitution, 
the natural suspension tiling satisfies the Pisot family condition and
our potential overlap algorithm readily applies. 
It is of interest to check this non-weakly mixing system, 
as it is shown by \cite{AvilaForni:07} that 
a generic interval exchange transformation which is not a rotation
is weakly mixing.
Our algorithm
shows that its suspension tiling dynamics is not pure discrete.
Although there are several known reducible Pisot 
substitutions which are not pure discrete,
this example is noteworthy in the sense that the  
non-coincident component is much more intricate
than in other known examples. 
We describe it by another substitution
on the 12 letters:
\begin{eqnarray}
\label{Ov2}
\nonumber
a&\rightarrow& afdeEafdc, \\
\nonumber
b&\rightarrow& fBCFbcfdeE, \\
\nonumber
c&\rightarrow& afdc, \\
\nonumber
A&\rightarrow& AFDEeAFDC, \\
\nonumber
B&\rightarrow& FbcfBCFDEe, \\
C&\rightarrow& AFDC, \\
\nonumber
d&\rightarrow& FbcfdeE, \\
\nonumber
D&\rightarrow& fBCFDEe, \\
\nonumber
e&\rightarrow& afdeEe, \\
\nonumber
E&\rightarrow& AFDEeE, \\
\nonumber
f&\rightarrow& fBC, \\
\nonumber
F&\rightarrow& Fbc
\end{eqnarray}
The associated tile equation of this substitution
gives the non-coincident component of the suspension tiling corresponding to 
(\ref{4IET}). 
In other words, (\ref{Ov2}) is the exact analogue of Figure \ref{Overlap},
once we associate the intervals of canonical lengths given by 
the left eigenvector of the substitution matrix. 
Since the original tiling is a factor of this new
suspension tiling, the system is again not pure discrete.


\medskip

The above two substitution examples have reducible characteristic polynomials. 
So the assumption of the Pisot substitution conjecture in higher
dimensions of \cite{ABBLS} does not hold.

\section*{Acknowledgment}

\noindent
We are grateful to P.~Arnoux, Sh.~Ito and M.~Furukado
for permitting us to include the last example for the study of its spectrum. 
This work was supported by the National Research Foundation of Korea(NRF) Grant funded by the Korean Government(MSIP)(2014004168), the
Japanese Society for the Promotion of Science (JSPS), Grant in aid
21540012, and the German Research Foundation (DFG) through the 
CRC 701 \emph{Spectral Structures and Topological Methods in Mathematics}.
The first author is partially supported by the T\'AMOP-4.2.2.C-11
/1/KONV-2012-0001 project. (The project is implemented through the New 
Hungary Development Plan, co-financed by the European Social Fund and the 
European Regional Development Fund.)
The third author is grateful for the support of Korea Institute for Advanced Study(KIAS) for this research.

\vspace{6mm}

{\footnotesize
a: 
Institute of Mathematics, University of Tsukuba, 1-1-1 Tennodai, \\
 \hspace*{2.5em} Tsukuba, Ibaraki, Japan (zip:350-0006);
\email{akiyama@math.tsukuba.ac.jp}

\smallskip
b: Faculty of Mathematics, Bielefeld University, \\
\hspace*{2.5em} Postfach 100131, 33501 Bielefeld, Germany;
\email{gaehler@math.uni-bielefeld.de}

\smallskip
c: Dept. of Math. Edu., Catholic Kwandong University, 24, 579 Beon-gil, Beomil-ro, Gangneung,  \\  
\hspace*{2.5em} Gangwon-do, 210-701 Republic of Korea;
\email{jylee@cku.ac.kr} }

\end{document}